\documentclass[a4paper]{amsart}

\usepackage[T1]{fontenc}
\usepackage[utf8]{inputenc}
\usepackage{amsmath, amssymb}
\usepackage[margin=2.0cm]{geometry}
\usepackage[usenames,dvipsnames]{color}
\usepackage{hyperref}
\usepackage[shortlabels]{enumitem}
\usepackage{xcolor}

\theoremstyle{plain}
\newtheorem{theorem}{Theorem}[section]
\newtheorem{proposition}[theorem]{Proposition}
\newtheorem{lemma}[theorem]{Lemma}
\newtheorem{corollary}[theorem]{Corollary}

\theoremstyle{definition}

\newtheorem{question}[theorem]{Question}
\theoremstyle{remark}
\newtheorem{remark}[theorem]{Remark}

\newcommand{\Lring}{\mathfrak{L}_{\rm ring}}
\newcommand{\Lval}{\mathfrak{L}_{\rm val}}

\newcommand{\URt}{\ref{URt}}
\newcommand\Rfour{\hyperref[R4]{{\rm(R4)}}}

\numberwithin{equation}{section}

\title{On the existential theory of the completions of a global field}

\author{Philip Dittmann and Arno Fehm}
\address{Institut für Algebra, Technische Universität Dresden, 01062 Dresden, Germany}
\curraddr{Department of Mathematics, University of Manchester, Manchester M13 9PL, United Kingdom}
\email{philip.dittmann@manchester.ac.uk}
\address{Institut für Algebra, Technische Universität Dresden, 01062 Dresden, Germany}
\email{arno.fehm@tu-dresden.de}

\begin{document}

\begin{abstract}
 We discuss the common existential theory of all or almost all completions of a global function field.
\end{abstract}

\maketitle

\section{Introduction}

\noindent
The celebrated Hasse--Minkowski local-global principle states that a quadratic form over a global field $K$ has a nontrivial zero in $K$ if and only if it has a nontrivial zero in each completion of $K$.
Given a concrete quadratic form $f\in K[x_1,\dots,x_n]$ one can effectively decide whether $f$ has a nontrivial zero in all completions, 
eventually leading to an algorithm to decide 
whether $f$ has a nontrivial zero in $K$.

For polynomials of higher degree the local-global principle fails, 
and it is conjectured (for $K$ a number field, see e.g.~\cite{Koenigsmann_survey,Poonen_survey})
respectively known (for $K$ a global function field \cite{Pheidas,Videla,Shlapentokh,Eisentraeger}) that
the problem of whether polynomials over $K$ have
a zero in $K$ is undecidable.
In this short note we want to study the local problem, i.e.~deciding whether a polynomial over $K$ has a zero in each completion of $K$.

For number fields $K$, even much more is known
by the seminal work of Ax:

\begin{theorem}[\cite{Ax}]\label{thm:intro_Ax}
The (common) $\mathfrak{L}_{\rm val}$-theory of all $\mathbb{Q}_p$
and the $\mathfrak{L}_{\rm val}$-theory of almost all $\mathbb{Q}_p$ {are both} decidable.
\end{theorem}

Here, $\mathfrak{L}_{\rm val}$ is the language of valued fields,
and by the theory of almost all $\mathbb{Q}_p$ we mean
the set of sentences true in all but finitely many $\mathbb{Q}_p$
(alternatively, here and in the following, we could interpret ``almost all'' in the sense of Dirichlet density $1$,
see Remark \ref{rem:density}).
Since also the $\mathfrak{L}_{\rm ring}$-theories of $\mathbb{R}$ and $\mathbb{C}$ are decidable by the work of Tarski \cite{Tarski}, and every number field can be interpreted in $\mathbb{Q}$, this leads to:

\begin{corollary}\label{cor:intro_Ax}
Let $K$ be a number field.
The $\mathfrak{L}_{\rm ring}$-theory of all completions of $K$
and the $\mathfrak{L}_{\rm ring}$-theory of almost all completions of $K$ are both decidable.
In particular, it is decidable whether a given polynomial over $K$ has a zero in all respectively almost all completions of $K$.\footnote{For a sketch see Remark \ref{rem:finite_ext}.}
\end{corollary}

See also \cite{Kostas} for a recent negative result in this direction.

For completions of global function fields, there is at this point little hope to decide the common theory of all completions, 
since even the theory of a single completion $\mathbb{F}_q(\!(t)\!)$ seems out of reach,
in spite of progress such as \cite{Kuhlmann}.
For the existential theory (that is, essentially, deciding which polynomials have zeros), the following result of Anscombe and the second author 
{
(which can be obtained by combining Corollary 5.11(1c,d) and Proposition 5.9(1a) of \cite{AF23})
is related but does not consider the family of completions of one fixed global function field like $\mathbb{F}_p(t)$, but instead varies $p$:}

\begin{theorem}[\cite{AF23}]\label{thm:intro1}
The existential $\mathfrak{L}_{\rm val}$-theory of all 
{
$\mathbb{F}_p(\!(t)\!)$ and the existential $\mathfrak{L}_{\rm val}$-theory of almost all $\mathbb{F}_p(\!(t)\!)$,}
where $p$ runs over the prime numbers, are both decidable.    
\end{theorem}

This builds crucially on our understanding of the existential theory of each individual $\mathbb{F}_q(\!(t)\!)$, where we have two kinds of results:
By \cite{AF16}, building on Kuhlmann's theory of tame valued fields \cite{Kuhlmann_tame},  the existential theory of any $\mathbb{F}_q(\!(t)\!)$ in the language of valued fields $\mathfrak{L}_{\rm val}$ is decidable.
The existential theory of $\mathbb{F}_q(\!(t)\!)$ in the language $\mathfrak{L}_{\rm val}(t)$ of valued fields with a constant symbol for $t$ was shown to be decidable in \cite{DS} assuming resolution of singularities in positive characteristic,
and in \cite{ADF23} assuming a consequence of local uniformisation (which in turn follows from resolution of singularities) called (R4) (see Section \ref{sec:fullparam}),
see also \cite{Kartas} for related recent results. Building on 
and extending these results, we will show:

\begin{theorem}\label{thm:intro_noparam}
For a global function field $K$, the following theories are decidable:
\begin{enumerate}[$(a)$]
\item the existential $\mathfrak{L}_{\rm val}$-theory of {\em all} completions of $K$,
\item the existential $\mathfrak{L}_{\rm val}$-theory of {\em almost all} completions of $K$,
\item the existential $\mathfrak{L}_{\rm val}(K)$-theory of {\em almost all} completions of the perfect hull of $K$,
\item the existential $\mathfrak{L}_{\rm ring}(K)$-theory of {\em almost all} completions of $K$.
\end{enumerate}
Moreover, assuming \Rfour\ holds, also the following theories are decidable:
\begin{enumerate}[$(a)$]
\setcounter{enumi}{4}
\item the existential $\mathfrak{L}_{\rm val}(K)$-theory of {\em all} completions of $K$,
\item the existential $\mathfrak{L}_{\rm val}(K)$-theory of {\em almost all} completions of $K$.
\end{enumerate}
\end{theorem}

Here, $\Lval(K)$ and $\Lring(K)$ are the expansions of $\Lval$ and $\Lring$ by constant symbols for the elements of $K$.

\begin{corollary}
For a global function field $K$,
it is decidable whether a polynomial over $K$ has a zero in all respectively almost all completions of $K$,
where in the case of ``all completions''
we need to either restrict to polynomials over the prime field of $K$, or assume that \Rfour\ holds.
\end{corollary}

In fact, we will not only prove decidability but also axiomatise these theories,
where possible for general function fields of one variable.
Moreover we will prove the results not just for existential sentences
but for 
universal/existential sentences,
i.e.~finite boolean combinations of existential sentences
(so that in particular we can decide not only whether a given polynomial 
{
has a zero in all completions
and whether it has a zero in almost all completions},
but also whether 
{
it has a zero in at most finitely many completions,
and whether it has a zero in no completion.}

The six parts of Theorem \ref{thm:intro1} are proven in four sections
(Section \ref{sec:noparam}-\ref{sec:fullparam}),
where each section opens with a general result in the model theory of valued fields.
While in the first of these sections (in which parts (a) and (b) are proven)
this result is readily available in the literature,
they get more and more novel as the paper progresses.
In particular, in Section \ref{sec:fullparam} we offer some generalizations to results from \cite{ADF23}, by replacing a separability condition with a new and weaker axiom regarding ramification.
In Section \ref{sec:notation} we fix notation and summarise the common arguments used in these four settings.

\section{Notation and preliminaries}
\label{sec:notation}

\noindent
A {\em valued field} is a pair $(F,v)$ consisting of a field $F$ and a (Krull) valuation $v$ on $F$ (with value group written additively).
For basics on valued fields see \cite{EP}.
By a {\em function field} (of one variable) we mean a finitely generated field extension $K/k$ of transcendence degree $1$
such that $k$ is algebraically closed in $K$.
We denote by $\mathbb{P}_K$ the set of places of $K/k$,
which we view as (normalised) discrete valuations on $K$, trivial on $k$.
Recall that for each $v\in\mathbb{P}_K$,
the residue field $Kv$ is a finite extension of $k$. 
A {\em global function field} is a function field $K/k$ with $k$ finite.
For basics on function fields see \cite{Stichtenoth}.
A field is {\em pseudofinite} if it is an infinite model of the theory of finite fields.
A {\em $K$-variety} is a separated scheme of finite type over $K$.

We work in the language $\mathfrak{L}_{\rm ring}=\{+,\cdot,0,1\}$ of rings,
and the one-sorted language $\mathfrak{L}_{\rm val}=\mathfrak{L}_{\rm ring}\cup\{\mathcal{O}\}$
of valued fields. For a valuation $v$ on a field $F$ we denote by
$F$ the $\mathfrak{L}_{\rm ring}$-structure $F$,
and by $(F,v)$ the corresponding $\mathfrak{L}_{\rm val}$-structure, where we interpret $\mathcal{O}$
as the valuation ring $\mathcal{O}_v$ of $v$.
We will always fix a field $C$ and view extensions $F$ of $C$ as
structures in the language $\mathfrak{L}_{\rm ring}(C)$, which is $\mathfrak{L}_{\rm ring}$
expanded by constant symbols for the elements of $C$, similarly for $\mathfrak{L}_{\rm val}(C)$.
In particular, 
${\rm Th}(F)$ is then the $\mathfrak{L}_{\rm ring}(C)$-theory of the field $F$ 
and ${\rm Th}(F,v)$ denotes the $\mathfrak{L}_{\rm val}(C)$-theory of the valued field $(F,v)$.
We view the residue field $Fv$ as an $\mathfrak{L}_{\rm ring}(C)$-structure via the residue map, which for definiteness we define to be constant 0 outside the valuation ring of $v$.
If $v$ is trivial on $C$, then we view $Fv$ again as an extension of $C$.
When we speak of decidability we always
assume that $C$ is countable and comes with a fixed injection $C\rightarrow\mathbb{N}$ such that, with a standard Gödel coding, the atomic diagram of $C$ is computable.
Such an injection can always be found when $C$ is a global function field (cf.~\cite[\S19.2]{FJ}).

An {\em existential} respectively {\em universal} formula is a formula of the form $\exists x_1,\dots,x_n\psi$ 
respectively  $\forall x_1,\dots,x_n\psi$
with $\psi$ quantifier-free.
A {\em universal/existential} formula is a formula obtained from finitely many existential or universal sentences by conjunctions and disjunctions.
For a theory $T$ we write $T_{\forall/\exists}$ for the set of universal/existential consequences of $T$ (in the language of $T$),
and we call ${\rm Th}_{\forall/\exists}(M):={\rm Th}(M)_{\forall/\exists}$ the universal/existential theory of the structure $M$.

We fix a field $C$
and denote by $T_0$ the $\mathfrak{L}_{\rm val}(C)$-theory of valued fields extending the trivially valued field $C$.
For an $\mathfrak{L}_{\rm val}(C)$-theory $T\supseteq T_0$ and an $\mathfrak{L}_{\rm ring}(C)$-theory $R$ we denote by
$T(R)$ the theory of $(F,v)\models T$ with $Fv\models R$.
Here we view $C$ as embedded into $Fv$ via the residue map.
For example, in the following lemma,
$T({\rm Th}_{\forall/\exists}(Fv))_{\forall/\exists}$ refers to the set of universal/existential consequences of the common $\mathfrak{L}_{\rm val}(C)$-theory of those $(L,w)\models T$ whose residue field $Lw$ satisfies the same
universal/existential $\mathfrak{L}_{\rm ring}(C)$-sentences as $Fv$.

\begin{lemma}\label{lem:main}
Let $T\supseteq T_0$ be an $\mathfrak{L}_{\rm val}(C)$-theory.
Suppose that
\begin{equation}\label{eqn:star}
\tag{$\ast$} {\rm Th}_{\forall/\exists}(F,v)={\rm Th}_{\forall/\exists}(F',v') \mbox{ whenever }  (F,v),(F',v')\models T \mbox{ with } {\rm Th}_{\forall/\exists}(Fv)={\rm Th}_{\forall/\exists}(F'v').
\end{equation}
Then, with $R_0={\rm Th}(\{Fv:(F,v)\models T\})$, the following holds:
\begin{enumerate}[$(a)$]
\item ${\rm Th}_{\forall/\exists}(F,v)=T({\rm Th}_{\forall/\exists}(Fv))_{\forall/\exists}$ for every $(F,v)\models T$.
\item If
$T$ is computably enumerable and $R\supseteq R_0$ is a decidable $\mathfrak{L}_{\rm ring}(C)$-theory,
then $T(R)_{\forall/\exists}$ is decidable\footnote{Recall the assumptions we have made on $C$ above when we speak about decidability.}.
\item If $(R_i)_{i\in I}$ is a family of deductively closed
$\mathfrak{L}_{\rm ring}(C)$-theories containing $R_0$, then
$\bigcap_{i\in I}T(R_i)_{\forall/\exists}=
T(\bigcap_{i\in I}R_i)_{\forall/\exists}$.
\end{enumerate}
Moreover, if we weaken the conclusion of $(*)$
to ${\rm Th}_{\forall/\exists}(F)={\rm Th}_{\forall/\exists}(F')$, then (a), (b) and (c) hold if we interpret each ``$\forall/\exists$'' to refer only to universal/existential $\mathfrak{L}_{\rm ring}(C)$-sentences.
\end{lemma}

\begin{proof}
This can be seen directly, or deduced from
Corollary 2.23 of \cite{AF23}:
The universal/existential sentences of
$\mathfrak{L}_1=\mathfrak{L}_{\rm ring}(C)$ and
$\mathfrak{L}_2=\mathfrak{L}_{\rm val}(C)$
form {\em fragments} in the terminology of that paper,
and together with the theories $R_0$ respectively $T$ and the map $\sigma\colon(F,v)\mapsto Fv$, they form a {\em bridge} $B$.
With $\hat{B}$ the same bridge but now {\em all}
$\mathfrak{L}_{\rm ring}(C)$- respectively
$\mathfrak{L}_{\rm val}(C)$-sentences,
and $\iota$ the map on sentences obtained from interpreting the residue field in the valued field, we obtain an {\em arch} $A=(B,\hat{B},\iota)$.
Assumption 
(i) of that corollary is satisfied
by our hypotheses on $C$ and $T$,
assumption (ii) is satisfied by the definition of $R_0$,
and assumption (iii) is satisfied by {\rm(\ref{eqn:star})}.
Part (a) of our lemma follows from (III) of the corollary,
part (b) follows from (II),
and part (c) follows from (I) together with
\cite[Lemma 2.24]{ADF23}.
The ``moreover'' part follows just the same, after changing $\mathfrak{L}_2$
to
$\mathfrak{L}_{\rm ring}(C)$.
\end{proof}

For the convenience of the reader, we now list the theories $T\supseteq T_0$ which we will discuss in this work:
\begin{enumerate}
    \item[$T_1$] is the $\mathfrak{L}_{\rm val}$-theory of equicharacteristic henselian nontrivially valued fields of fixed characteristic $p\geq 0$, see Section \ref{sec:noparam}.
    \item[$T_2$] is the $\mathfrak{L}_{\rm val}(C)$-theory of henselian nontrivially valued fields $(F,v)$ extending the trivially valued field $C$ with $Fv$ perfect and $Fv/C$ separable, see Section \ref{sec:4}.
    \item[$T_3$] is the $\mathfrak{L}_{\rm val}(C)$-theory
of henselian nontrivially valued fields $(F,v)$ extending the trivially valued field $C$ such that $F/C$ is separable and every geometrically integral smooth $C$-variety has $Fv$-rational points, see Section \ref{sec:5}.
\item[$T_4$] is the $\mathfrak{L}_{\rm val}(C)$-theory
of henselian nontrivially valued fields $(F,v)$
extending the trivially valued field $C$
satisfying an additional axiom \URt, to be defined in Section \ref{sec:fullparam}.
\end{enumerate}

\section{Existential $\mathfrak{L}_{\rm val}$-theory without parameters}
\label{sec:noparam}

\noindent
We start by studying the universal/existential $\mathfrak{L}_{\rm val}$-theory of all and almost all completions of a function field without any parameters.
Here our background theory is the $\mathfrak{L}_{\rm val}$-theory $T_1$ of equicharacteristic henselian nontrivially valued fields of fixed characteristic $p\geq 0$, which for technical reasons we view as an $\mathfrak{L}_{\rm val}(C)$-theory for $C=\mathbb{F}_p$ respectively $C=\mathbb{Q}$,
{and the main result we will use about this theory is the following:

\begin{theorem}[{\cite[Thm.~1.1]{AF16}}]\label{thm:AF16_1.1}
For every field $k$, the theory
$T=T_1({\rm Th}_{\forall/\exists}(k))$ is $\exists$-complete,
i.e.~for every existential $\mathfrak{L}_{\rm val}$-sentence $\varphi$ we have either $T\models\varphi$
or $T\models\neg\varphi$.
\end{theorem}
}

\begin{proposition}\label{lem:He}
If $(F,v)$ and $(F',v')$ are equicharacteristic henselian nontrivially valued fields
whose residue fields $Fv$ and $F'v'$ have the same universal/existential $\mathfrak{L}_{\rm ring}$-theory, then
$(F,v)$ and $(F',v')$ have the same 
universal/existential $\mathfrak{L}_{\rm val}$-theory.
In particular, {\rm(\ref{eqn:star})} holds for $T=T_1$ with $C$ a prime field.
\end{proposition}

\begin{proof}
By {Theorem \ref{thm:AF16_1.1}}, 
$(F,v)$ and $(F',v')$ have the same existential $\mathfrak{L}_{\rm val}$-theory,
and therefore also the same universal/existential $\mathfrak{L}_{\rm val}$-theory.
\end{proof}

If $K/k$ is a function field and $v\in\mathbb{P}_K$, we denote by
$(\hat{K}_v,\hat{v})$ the completion of $(K,v)$,
which is an equicharacteristic henselian nontrivially valued field
with residue field $\hat{K}_v\hat{v}=Kv$.

\begin{proposition}\label{prop:Eax}
Let $K/k$ be a function field of one variable.
The universal/existential $\mathfrak{L}_{\rm val}$-theory of all respectively almost all completions $\hat{K}_v$, $v\in\mathbb{P}_K$,
is the universal/existential theory of equicharacteristic henselian nontrivially valued fields with residue field a
model of the theory of all respectively almost all $Kv$.
\end{proposition}

\begin{proof}
As {\rm(\ref{eqn:star})}  holds for $T=T_1$ by Proposition \ref{lem:He},
and $(\hat{K}_v,\hat{v})$ is a model of $T_1$,
Lemma \ref{lem:main}(a) gives that
${\rm Th}_{\forall/\exists}(\hat{K}_v,\hat{v})=T_1({\rm Th}(Kv))_{\forall/\exists}$ for every $v\in\mathbb{P}_K$.
Lemma \ref{lem:main}(c) shows that
for each $P\subseteq\mathbb{P}_K$,
\begin{equation}\label{eqn:intersection}
 \bigcap_{v\in P}T_1({\rm Th}(Kv))_{\forall/\exists}\;=\;T_1(\bigcap_{v\in P}{\rm Th}(Kv))_{\forall/\exists}.
\end{equation}
This immediately implies the claim for {\em all} completions by taking $P=\mathbb{P}_K$ in (\ref{eqn:intersection}), and for {\em almost all} completions one gets that
$$
 \bigcup_{P_0\subseteq\mathbb{P}_K\atop\mbox{\tiny finite}}\bigcap_{v\in\mathbb{P}_K\setminus P_0}T_1({\rm Th}(Kv))_{\forall/\exists} 
 \;=\; \bigcup_{P_0\subseteq\mathbb{P}_K\atop\mbox{\tiny finite}}T_1\big(\bigcap_{v\in\mathbb{P}_K\setminus P_0}{\rm Th}(Kv)\big)_{\forall/\exists}
 \;=\; T_1\big(\bigcup_{P_0\subseteq\mathbb{P}_K\atop\mbox{\tiny finite}}\bigcap_{v\in\mathbb{P}_K\setminus P_0}{\rm Th}(Kv)\big)_{\forall/\exists},
$$
where the first equality holds by (\ref{eqn:intersection})
and the second equality holds trivially.
\end{proof}

We now turn to the case where $k=\mathbb{F}_q$ is finite.
Then $Kv$ is a finite field
for each $v\in\mathbb{P}_K$,
and the completion $(\hat{K}_v,\hat{v})$ of $(K,v)$ is isomorphic
to the Laurent series field $(Kv(\!(t)\!),v_t)$ with the $t$-adic valuation $v_t$ \cite[Thm.~II.2]{Serre}.
For $f\in k[t,x]$ irreducible we denote
by $K_f$ the global function field ${\rm Quot}(k[t,x]/(f))$.

\begin{lemma}\label{lem:resfields}
If $k=\mathbb{F}_q$, 
the map $v\mapsto|Kv|$ has finite fibres and
$$
 R(K):=\{|Kv|:v\in\mathbb{P}_K\}
$$ 
is a cofinite subset of
$\{q^n:n\in\mathbb{N}\}$.
More precisely, if $K=K_f$ with $f\in\mathbb{F}_q[t,x]$ absolutely irreducible of total degree $\deg(f)$, then
$$
E(K):=\{q^n: n\in\mathbb{N},q^n\notin R(K)\}\subseteq\{q,q^2,\dots,q^{2{\rm deg}(f)^2}\}.
$$
\end{lemma}

\begin{proof}
Every $Kv$ is a finite extension of the field of constants $k$, and so $|Kv|=q^{[Kv:k]}$.
Without loss of generality, $d:={\rm deg}(f)>1$.
The genus $g$ of $K$ is at most $(d-1)^2/2$ \cite[Cor.~3.11.4]{Stichtenoth},
and the Hasse--Weil bound 
implies that 
$K$ has a place of degree $n$
whenever $n\geq 4g+3$
\cite[Cor.~5.2.10(c)]{Stichtenoth}.
This holds in case $n>2d^2$.
\end{proof}

\begin{theorem}\label{thm:Eax}
Let $K/k$ be a global function field.
\begin{enumerate}[(a)]
\item 
The universal/existential theory of {\em almost all} completions $\hat{K}_v$
is the universal/existential theory of
equicharacteristic henselian nontrivially valued fields with residue field a pseudofinite field extending~$k$.
\item 
The universal/existential theory of {\em all} completions $\hat{K}_v$
is the universal/existential theory of
equicharacteristic henselian nontrivially valued fields whose residue fields
are finite or pseudofinite, extend $k$,
and satisfy the additional sentence 
$\eta_K:=\bigwedge_{q'\in E(K)}\neg\chi_{q'}$,
where $\chi_{q'}$ is such that $M\models\chi_{q'}\Leftrightarrow M\cong\mathbb{F}_{q'}$.
\end{enumerate}
\end{theorem}

\begin{proof}
(a) follows immediately from Proposition \ref{prop:Eax}, since the theory of almost all $Kv$ is by Lemma \ref{lem:resfields} the theory of almost all finite extensions of $k\cong\mathbb{F}_q$.
So if $\psi_q$ is a sentence expressing that the field contains a subfield isomorphic to $\mathbb{F}_q$, then a sentence $\varphi$ holds in almost all $Kv$ if and only if $\psi_q\rightarrow\varphi$ holds in almost all finite fields.
By \cite[Thm.~9]{Ax} this is the case if and only if it holds in all pseudofinite fields, i.e.~$\varphi$ holds in all pseudofinite fields extending $\mathbb{F}_q$.

For (b), similarly a sentence $\varphi$ holds in all $Kv$ if and only if
$(\psi_q\wedge\eta_K)\rightarrow\varphi$ holds in all finite fields.
By \cite[Thm.~9]{Ax} this is the case if and only if it holds in all finite or pseudofinite fields, i.e.~$\varphi$ holds in all finite or pseudofinite fields extending $\mathbb{F}_q$ that in addition satisfy $\eta_K$.
\end{proof}

\begin{corollary}\label{cor:Edec}
There is an algorithm which, given
a prime power $q$,
an absolutely irreducible polynomial $f\in\mathbb{F}_q[t,x]$
and a universal/existential $\mathfrak{L}_{\rm val}$-sentence $\varphi$
decides whether $\varphi$
holds in {\em all} completions of $K_f$,
and whether $\varphi$
holds in {\em almost all} completions of $K_f$.
\end{corollary}

\begin{proof}
By Proposition \ref{lem:He} and Lemma \ref{lem:main}(b), 
$T_1(R)_{\forall/\exists}$ is decidable
for any decidable $\mathfrak{L}_{\rm ring}$-theory $R$.
So by Theorem \ref{thm:Eax} it suffices to show that the theory of pseudofinite fields extending $\mathbb{F}_q$,
and the theory of fields that are finite or pseudofinite extending $\mathbb{F}_q$ and satisfying $\eta_K$ are decidable.
They indeed are,
since the theory of finite fields and the theory of pseudofinite fields are decidable by Ax \cite[Thm.~13, 13']{Ax},
the map $q\mapsto\psi_q$ 
(defined as in the proof of Theorem \ref{thm:Eax}) is obviously computable,
and also the map $(q,f)\mapsto \eta_{K_f}$ is computable:
By enumerating varieties over $\mathbb{F}_q$ and rational maps we can effectively find a smooth projective curve $X$ over $\mathbb{F}_q$ birationally equivalent to the affine curve $f=0$,
and then $q^n\in E(K)$ if and only if
every $\mathbb{F}_{q^n}$-rational point of $X$
is $\mathbb{F}_{q^m}$-rational for some $m|n$, $m\neq n$,
which we can check for the finitely many
$n\leq2{\rm deg}(f)^2$.
\end{proof}

Theorem \ref{thm:intro_noparam}(a) and (b) are an immediate consequence of Corollary~\ref{cor:Edec} in which we fix $q$ and $f$ with $K\cong K_f$.

\begin{remark}
The algorithm described in Corollary \ref{cor:Edec} can be extended to output,
whenever $\varphi$ holds in almost all completions of $K=K_f$, 
a list of precisely those finitely many $v\in\mathbb{P}_K$ such that $(\hat{K}_v,\hat{v})\not\models\varphi$.
Since making this precise would first require a choice of presentation of the places of $K$, we omit the details.
\end{remark}

\begin{remark}
If (like in Theorem \ref{thm:intro_noparam}) one is interested only in the {\em existential} theory
(rather than the universal/existential theory),
the description in Theorem~\ref{thm:Eax}(b) and its proof can be simplified further:
For example, if $K=\mathbb{F}_q(t)$, or more generally if $K$ has a place of degree one, then the existential theory of all $Kv$
is just the existential theory of $\mathbb{F}_q$.
In this case, the existential theory of all $\hat{K}_v$ is
the existential theory of $\mathbb{F}_q(\!(t)\!)$.
Similarly, the {\em universal} theory of all (or equivalently almost all) $Kv$ equals the universal theory of the algebraic closure $\overline{\mathbb{F}}_q$,
and one can show that the universal theory of all (or almost all) $\hat{K}_v$
is the universal theory of $\overline{\mathbb{F}}_q(\!(t)\!)$.
\end{remark}

\begin{remark}\label{rem:define_valuation}
One could deduce 
Theorem \ref{thm:intro_noparam}(a,b)
from the (a priori weaker) corresponding statement for $\mathfrak{L}_{\rm ring}$
instead of $\mathfrak{L}_{\rm val}$
if there were an existential $\mathfrak{L}_{\rm ring}$-formula $\varphi(x)$ and a universal $\mathfrak{L}_{\rm ring}$-formula $\psi(x)$
(all formulas here without parameters)
such that both $\varphi$ and $\psi$ would define the valuation ring in all (respectively almost all) completions of $K$.
However, 
since by Lemma \ref{lem:resfields} the smallest elementary class $\mathcal{F}\supseteq\{Kv:v\in\mathbb{P}_K\}$ contains pseudofinite fields extending $\overline{\mathbb{F}}_q$,
\cite[Thm.~5.1, Prop.~6.5(a), 6.20(a)]{AF17} (with $C=\mathbb{Z}$, $\mathcal{K}=\{(\hat{K}_v,\hat{v}):v\in\mathbb{P}_K\})$
proves that neither $\varphi$ nor $\psi$ can exist.
We point out that although
it is known that there exist
$\exists\forall$-formulas 
and $\forall\exists$-formulas
that 
define the valuation ring in all completions of $K$
(see \cite[Thm.~1.4, 1.6]{FP},
and \cite{FehmJahnke} for more on definable valuations),
no universal/existential
formula $\eta$
defines the valuation ring in almost all completions of $K$:
Fix $v_0\in\mathbb{P}_K$
such that $q^{d}\in R(K)$ for every $q^d\geq|Kv_0|$
and such that $\eta$ defines the valuation ring in $\hat{K}_v$
for every $v$ with $|Kv|\geq |Kv_0|$.
Then $\eta$ defines the valuation ring of the unique extension of $v_0$ in every finite extension of
$\hat{K}_{v_0}$.
Since both $\eta$ and $\neg\eta$ are 
logically equivalent to $\forall\exists$-formulas,
and therefore go up chains \cite[Exercise 3.1.8]{ChangKeisler}, 
this implies that
also in the algebraic closure of $\hat{K}_{v_0}$,
$\eta$ defines the valuation ring of the unique extension of $v_0$,
contradicting quantifier elimination for algebraically closed fields.
\end{remark}

\section{Existential $\mathfrak{L}_{\rm val}$-theory with constant parameters, and the perfect hull}
\label{sec:4}

\noindent
Next we consider two cases where adding constants from a subfield $C$ of $K$ does not do much harm:
This essentially happens when the fields involved are perfect.
Our background theory now is
the $\mathfrak{L}_{\rm val}(C)$-theory $T_2\supseteq T_1$ of henselian nontrivially valued fields $(F,v)$ extending the trivially valued field $C$
with $Fv$ perfect and $Fv/C$ separable (which implies $C$ perfect).
{
We will use the following result about this theory:

\begin{theorem}[{\cite[Corollary 5.7]{AF16}}]\label{thm:AF16_5.7}
For every separable extension $k/C$ with $k$ perfect, the theory $T=T_2({\rm Th}(k))$ is $\exists$-$C$-complete,
i.e.~for any existential $\mathfrak{L}_{\rm val}(C)$-sentence $\varphi$
we have either $T\models\varphi$ or $T\models\neg\varphi$.
\end{theorem}
}

\begin{proposition}\label{lem:TC}
Let $C$ be perfect.
If $(F,v)$ and $(F',v')$
are henselian nontrivially valued fields extending the trivially valued field $C$
whose residue fields $Fv$ and $F'v'$ are perfect
and have the same universal/existential
$\mathfrak{L}_{\rm ring}(C)$-theory,
then
$(F,v)$ and $(F',v')$
have the same universal/existential $\mathfrak{L}_{\rm val}(C)$-theory,
i.e.~{\rm(\ref{eqn:star})} holds for $T=T_2$.
\end{proposition}

\begin{proof} 
Theorem \ref{thm:AF16_5.7} implies that
${\rm Th}_{\forall/\exists}(L,w)=T_2({\rm Th}(Lw))_{\forall/\exists}$ for every $(L,w)\models T_2$.
We assume ${\rm Th}_{\forall/\exists}(Fv)={\rm Th}_{\forall/\exists}(F'v')$
and have to show that ${\rm Th}_{\forall/\exists}(F,v)={\rm Th}_{\forall/\exists}(F',v')$,
and so we can replace $(F',v')$ by an elementary extension to
assume without loss of generality that $F'v'$ is $|Fv|^+$-saturated.
As in particular ${\rm Th}_\exists(Fv)\subseteq{\rm Th}_\exists(F'v')$
we obtain an embedding $Fv\rightarrow F'v'$ over $C$ \cite[Lemma 5.2.1]{ChangKeisler},
which extends to an embedding $(Fv(\!(t)\!),v_t)\rightarrow(F'v'(\!(t)\!),v_t)$,
proving 
$$
 {\rm Th}_\exists(F,v)=
{\rm Th}_\exists(Fv(\!(t)\!),v_t)\subseteq
{\rm Th}_\exists(F'v'(\!(t)\!),v_t)=
{\rm Th}_\exists(F',v'),
$$
where the two equalities follow from the result mentioned above, as $(Fv(\!(t)\!),v_t)$ and
$(F'v'(\!(t)\!),v_t)$ are models of $T_2$.
The symmetric argument shows that
${\rm Th}_\exists(F',v')\subseteq{\rm Th}_\exists(F,v)$,
hence 
${\rm Th}_{\forall/\exists}(F,v)={\rm Th}_{\forall/\exists}(F',v')$.
\end{proof}

\begin{proposition}
Let $K/k$ be a function field of one variable with $k$ perfect.
The universal/existential $\mathfrak{L}_{\rm val}(k)$-theory of {\em all} respectively {\em almost all} completions $\hat{K}_v$, $v\in\mathbb{P}_K$,
is the universal/existential $\mathfrak{L}_{\rm val}(k)$-theory of henselian nontrivially valued fields 
extending the trivially valued field $k$
with residue field a model of the $\mathfrak{L}_{\rm ring}(k)$-theory of all respectively almost all $Kv$.
\end{proposition}

\begin{proof}
Let $C=k$ and note that $k$ perfect implies that every $Kv$ is perfect, so $(\hat{K}_v,\hat{v})\models T_2$ for every $v\in\mathbb{P}_K$.
Therefore, this works completely analogous to the proof of Proposition \ref{prop:Eax},
replacing $T_1$ by $T_2$ and Proposition \ref{lem:He} by Proposition \ref{lem:TC}.
\end{proof}

For a field $K$ we denote by $K^{\rm perf}$ its maximal purely inseparable extension within $\overline{K}$.
For a valuation $v$ on a field $F$ we denote the unique extension of $v$ to $F^{\rm perf}$ again by $v$.

\begin{proposition}\label{prop:perf}
Let $K/k$ be a function field of one variable.
The following theories coincide:
\begin{enumerate}[(1)]
\item  The universal/existential $\mathfrak{L}_{\rm val}(K)$-theory of {\em almost all} $(\hat{K}_v)^{\rm perf}$, $v\in\mathbb{P}_K$.
\item The universal/existential $\mathfrak{L}_{\rm val}(K)$-theory of {\em almost all} $\hat{L}_v$, $v\in\mathbb{P}_L$, where $L=K^{\rm perf}$.
\item The universal/existential $\mathfrak{L}_{\rm val}(K)$-theory
of henselian nontrivially valued fields extending
the trivially valued field $K^{\rm perf}$
with residue field 
a model of the $\mathfrak{L}_{\rm ring}(K)$-theory
of almost all $(Kv)^{\rm perf}$.\footnote{Recall that we view $Kv$ as an $\mathfrak{L}_{\rm ring}(C)$-structure via the residue map $C\hookrightarrow K\rightarrow K_v$, where here $C=K$.}
\end{enumerate}
\end{proposition}  

\begin{proof}
Let $C=K^{\rm perf}$.
By Łoś's theorem \cite[Thm.~4.1.9]{ChangKeisler},
the $\mathfrak{L}_{\rm val}(C)$-theory of almost all $(\hat{K}_v)^{\rm perf}$
is the $\mathfrak{L}_{\rm val}(C)$-theory of the ultraproducts
$(K_\mathcal{U},v_\mathcal{U})=\prod_{v\in\mathbb{P}_K}((\hat{K}_v)^{\rm perf},\hat{v})/\mathcal{U}$,
where $\mathcal{U}$ runs over the nonprincipal ultrafilters on $\mathbb{P}_K$.
As every $((\hat{K}_v)^{\rm perf},\hat{v})$ is a perfect henselian nontrivially valued field,
so is $(K_\mathcal{U},v_\mathcal{U})$.
In particular, $K_\mathcal{U}v_\mathcal{U}$ is perfect.
Since every $c\in C^\times$ has nonzero value for at most finitely many $v\in\mathbb{P}_K$, $v_\mathcal{U}$ is trivial on $C$.
Thus $(K_\mathcal{U},v_\mathcal{U})\models T_2$.
Similarly,
the 
$\mathfrak{L}_{\rm val}(C)$-theory of almost all $\hat{L}_v$
is the $\mathfrak{L}_{\rm val}(C)$-theory of the ultraproducts
$(L_\mathcal{U},w_\mathcal{U})=\prod_{v\in\mathbb{P}_K}(\hat{L}_v,\hat{v})/\mathcal{U}$,
and also $(L_\mathcal{U},w_\mathcal{U})\models T_2$.
Moreover,
as $(\hat{K}_v)^{\rm perf}\hat{v}=(Kv)^{\rm perf}=Lv$,
we have that
$K_\mathcal{U}v_\mathcal{U}=\prod_{v\in\mathbb{P}_K}(Kv)^{\rm perf}/\mathcal{U}
=L_\mathcal{U}w_\mathcal{U}$.
Proposition \ref{lem:TC}
and Lemma \ref{lem:main}(a,c)
therefore give that
$$
  {\rm Th}_{\forall/\exists}(K_\mathcal{U},v_\mathcal{U})
  =
   T_2({\rm Th}(K_\mathcal{U}v_\mathcal{U}))_{\forall/\exists}
  =
  T_2({\rm Th}(L_\mathcal{U}w_\mathcal{U}))_{\forall/\exists}
  =
  {\rm Th}_{\forall/\exists}(L_\mathcal{U},w_\mathcal{U})
$$
and
$$
  \bigcap_{\mathcal{U}}{\rm Th}_{\forall/\exists}(K_\mathcal{U},v_\mathcal{U}) = 
 \bigcap_{\mathcal{U}} T_2({\rm Th}(K_\mathcal{U}v_\mathcal{U}))_{\forall/\exists} = 
 T_2(\bigcap_{\mathcal{U}}{\rm Th}(K_\mathcal{U}v_\mathcal{U}))_{\forall/\exists}.
$$
As
$\bigcap_{\mathcal{U}}{\rm Th}(K_\mathcal{U}v_\mathcal{U})=\bigcap_{\mathcal{U}}{\rm Th}(\prod_{v\in\mathbb{P}_K} (Kv)^{\rm perf}/\mathcal{U})$
is the $\mathfrak{L}_{\rm ring}(C)$-theory of almost all $(Kv)^{\rm perf}$,
restricting the languages from $\mathfrak{L}_{\rm val}(C)$ and $\mathfrak{L}_{\rm ring}(C)$ to $\mathfrak{L}_{\rm val}(K)$ and
$\mathfrak{L}_{\rm ring}(K)$
gives the claim.
\end{proof}

\begin{lemma}\label{lem:psf}
For a global function field $K$,
the $\mathfrak{L}_{\rm ring}(K)$-theory of almost all $Kv$, $v\in\mathbb{P}_K$,
is the $\mathfrak{L}_{\rm ring}(K)$-theory of pseudofinite fields extending $K$,
and it is decidable.
\end{lemma}

\begin{proof}
Let $k$ be the constant field of $K$, fix a separating element $t\in K$,
let $\mathcal{O}_K$ be the integral closure of $k[t]$ in $K$,
and let $P\subseteq\mathbb{P}_K$ be the set of places $v$ of $K$ with $\mathcal{O}_K\subseteq \mathcal{O}_v$.
As $\mathbb{P}_K\setminus P$ is the set of poles of $t$,
and hence finite,
the $\mathfrak{L}_{\rm ring}(K)$-theory of almost all $Kv$, $v\in\mathbb{P}_K$,
equals the
$\mathfrak{L}_{\rm ring}(K)$-theory of almost all $Kv$, $v\in P$.
By \cite[Theorem 20.9.4]{FJ},
this is the set of all 
$\mathfrak{L}_{\rm ring}(K)$-sentences
for which $\{v\in P:Kv\models\varphi\}$ has Dirichlet density $1$, and it is decidable\footnote{Note
    that the theorem referenced uses $\Lring(\mathcal{O}_K)$ instead of $\Lring(K)$, but these languages are equally expressive.}.
Moreover,
by \cite[Theorems 20.9.3]{FJ}
the latter theory
is the set of $\mathfrak{L}_{\rm ring}(K)$-sentences $\varphi$ 
for which the set of $\sigma$ 
in ${\rm Aut}(\overline{K}/K)$
(which is a profinite group canonically isomorphic to the absolute Galois group $G_K$)
such that $\varphi$ holds in the fixed field of $\sigma$ has Haar measure $1$.
By \cite[Theorem 20.5.4]{FJ} (with $e=1$),
this is precisely the theory of perfect PAC fields $F$ extending $K$ with $G_F\cong\hat{\mathbb{Z}}$, the free profinite group on $e=1$ generators.
As a field $F$ is pseudofinite if and only if it is perfect and PAC with $G_F\cong\hat{\mathbb{Z}}$ (see \cite[Corollary 20.10.5]{FJ}),
this proves the claim.
\end{proof}

\begin{theorem}\label{thm:perf}
For a global function field $K$, the universal/existential $\mathfrak{L}_{\rm val}(K)$-theory of {\em  almost all} $(\hat{K}_v)^{\rm perf}$, $v\in\mathbb{P}_K$,
is the universal/existential $\mathfrak{L}_{\rm val}(K)$-theory of henselian nontrivially valued fields extending
the trivially valued field $K^{\rm perf}$
with residue field 
pseudofinite,
and it is decidable.
\end{theorem}

\begin{proof}
As each $Kv$ is finite and therefore $Kv=(Kv)^{\rm perf}$,
by Lemma \ref{lem:psf}
the $\mathfrak{L}_{\rm ring}(K)$-theory of almost all $(Kv)^{\rm perf}$
is the theory of pseudofinite fields containing $K$.
Therefore, the first part follows from Proposition \ref{prop:perf}.
The decidability follows from Proposition \ref{lem:TC}, Lemma \ref{lem:psf} and Lemma \ref{lem:main}(b).
\end{proof}

Combining this with
the equivalence of (1) and (2) in Proposition \ref{prop:perf}
proves Theorem \ref{thm:intro_noparam}(c).

\begin{remark}\label{rem:density}
As noted in the introduction, we could always interpret ``almost all'' in the sense of Dirichlet density~$1$
(as defined for example in \cite[Chapter 6.3]{FJ} or \cite[p.~123]{Rosen}).
Indeed, for Lemma \ref{lem:psf}, this is explicit in its proof.
This then carries over to Theorem \ref{thm:Eax}
if in the proof of Proposition \ref{prop:Eax} one lets $P_0$ run over sets of Dirichlet density $0$
rather than only over finite sets.
It also carries over to Theorem \ref{thm:perf} (and similarly the theorems in the following two sections)
if in the proof of Proposition \ref{prop:perf} one lets $\mathcal{U}$ run only over those
ultrafilters on $\mathbb{P}_K$ that contain the filter of subsets of Dirichlet density $1$
rather than over all nonprincipal ultrafilters.
\end{remark}

\begin{remark}\label{rem:all_implies_individual}
We are not able to prove that the universal/existential $\mathfrak{L}_{\rm val}(K)$-theory of {\em all} $(\hat{K}_v)^{\rm perf}$ is decidable.
This is, however, not surprising:  
Such an algorithm would, in particular, be able to 
solve the open problem of
deciding the existential $\mathfrak{L}_{\rm val}(K)$-theory of each individual $(\hat{K}_{v_0})^{\rm perf}$: If we fix a $z\in K$ whose only zero is $v_0$,
then $(\hat{K}_{v_0})^{\rm perf}\models\varphi$ if and only if
$z^{-1}\notin\mathcal{O}\rightarrow\varphi$
holds in all $(\hat{K}_v)^{\rm perf}$.
\end{remark}

\section{Existential $\mathfrak{L}_{\rm ring}$-theory with parameters}
\label{sec:5}

\noindent
We now return to the completions of $K$ and allow arbitrary parameters from $K$,
where we can obtain unconditional results when we restrict the language to $\mathfrak{L}_{\rm ring}$ (rather than $\mathfrak{L}_{\rm val}$). 
Here it will be crucial that a pseudofinite field $L$ is PAC, i.e.~every geometrically integral $L$-variety has $L$-rational points, cf.~the proof of Lemma~\ref{lem:psf}.

Let $T_3\supseteq T_1$ be the $\mathfrak{L}_{\rm val}(C)$-theory
of henselian nontrivially valued fields $(F,v)$ extending the trivially valued field $C$
such that $F/C$ is separable and every geometrically integral smooth $C$-variety has $Fv$-rational points.

{
\begin{proposition}\label{prop:Lring-params-henselian-comparison}
If $(F,v)$ and $(F',v')$ are henselian nontrivially valued separable extensions of a trivially valued field $C$ 
such that $Fv$ and $F'v'$
have the same universal/existential $\mathfrak{L}_{\rm ring}(C)$-theory
and every geometrically integral smooth affine $C$-variety has
a rational point over both $Fv$ and $F'v'$,
then $F$ and $F'$
have the same universal/existential $\mathfrak{L}_{\rm ring}(C)$-theory,
i.e.~{\rm(\ref{eqn:star})} holds for $T=T_3$ with the weaker conclusion for $\mathfrak{L}_{\rm ring}(C)$ instead of $\mathfrak{L}_{\rm val}(C)$.
\end{proposition}}
\begin{proof}
We claim that by enlarging $C$ we can assume without loss of generality that $C$ is separably closed in $Fv$ and $F'v'$.
Indeed, let $D$ be the separable algebraic closure of $C$ in $Fv$.
Since
$Fv$ and $F'v'$ have the same universal/existential $\Lring(K)$-theory,
the relative algebraic closure of $C$ in $Fv$ is isomorphic to the one in $F'v'$ \cite[Lemma 20.6.3(b)]{FJ},
and so we can choose an isomorphism  of $D$ with the separable algebraic closure of $C$ in  $F'v'$.
As $v$ is trivial on $C$, we can identify the separable algebraic closure of $C$ in $F$ with a subfield of $D$ via the residue map,
and since $v$ is henselian, this subfield is actually all of $D$ (cf.~\cite[Lemma 2.3]{AF16}). 
Similarly we identify $D$
with the separable algebraic closure of $C$ in $F'$.
Note that $F$ and $F'$ remain separable over $D$.
Finally, every geometrically integral smooth affine $D$-variety has a rational point over $Fv$ and $F'v'$:
Such a $D$-variety is the base change of a geometrically integral smooth affine $C'$-variety $V$ for a finite subextension $C'/C$ of $D/C$. 
The Weil restriction $W={\rm res}_{C'/C}(V)$ is a geometrically integral smooth affine $C$-variety (see \cite[Prop.~4.9]{Scheiderer} and
\cite[Lemma 1.10]{Diem}),
hence $W(Fv)\neq\emptyset$ by assumption.
From the morphism $W_{C'}\rightarrow V$ (see \cite[4.2.5]{Scheiderer}) 
we conclude that also $V(Fv)\neq\emptyset$,
and similarly $V(F'v')\neq\emptyset$.

So from now on assume that $C$ is separably closed in $Fv$ and $F'v'$.
We already argued that then it
is also separably closed in $F$.
Since $F/C$ is in addition separable,
it is therefore regular.
Similarly, $F'/C$ is regular.
Therefore, by \cite[Lemma 4.2]{TwoExamples}, $F$ and $F'$ have the same existential $\Lring(C)$-theory, and therefore the same universal/existential $\Lring(C)$-theory, 
as soon as every geometrically integral smooth affine $C$-variety $V$ has an $F$-rational point if and only if it has an $F'$-rational point.
But every such $V$ has an $Fv$-rational point by assumption,
and this $Fv$-rational point lifts to an $F$-rational point of $V$ since $v$ is henselian
(see for instance \cite[\S2.3 Prop.~5]{BLR}),
and similarly to an $F'$-rational point. 
This concludes the proof.
\end{proof}

\begin{lemma}\label{lem:almost_all_hatKv}
Let $K$ be a global function field.
The $\mathfrak{L}_{\rm val}(K)$-theory of almost all $\hat{K}_v$
contains
the $\mathfrak{L}_{\rm val}(K)$-theory of 
henselian nontrivially valued fields
separable over the trivially valued field $K$
with residue field pseudofinite.
\end{lemma}

\begin{proof}
Each $(\hat{K}_v,\hat{v})$ is henselian nontrivially valued.
Furthermore, each $\hat{K}_v$ is separable over $K$,
since a uniformiser $\pi \in K$ for $v$ is a $p$-basis for both $K$ and $\hat{K}_v$.
For each $x\in K^\times$, $v(x)\neq0$ for only finitely many $v\in\mathbb{P}_K$,
so models of the $\mathfrak{L}_{\rm val}(K)$-theory of almost all $\hat{K}_v$ have valuation trivial on $K$.
Finally, such a model has pseudofinite residue field by Lemma \ref{lem:psf}.
\end{proof}

\begin{theorem}
\label{prop:Lring-params-aa}
  Let $K$ be a global function field.
  The universal/existential $\Lring(K)$-theory of almost all $\hat{K}_v$, $v\in\mathbb{P}_K$,
  is the universal/existential $\Lring(K)$-part of the
  $\Lval(K)$-theory of henselian nontrivially valued fields separable over the trivially valued field $K$ with residue field pseudofinite,
  and it is decidable.
\end{theorem}

\begin{proof}
Let $C=K$.
The $\mathfrak{L}_{\rm val}(K)$-theory of almost all $\hat{K}_v$ is the 
$\mathfrak{L}_{\rm val}(K)$-theory
of nonprincipal ultraproducts
$(K_\mathcal{U},v_\mathcal{U})=\prod_{v\in\mathbb{P}_K}(\hat{K}_v,\hat{v})/\mathcal{U}$.
By Lemma \ref{lem:almost_all_hatKv},
every such $(K_\mathcal{U},v_\mathcal{U})$ is a model of $T_3$:
Indeed, pseudofinite fields are PAC,
so even every geometrically integral 
$K_\mathcal{U}v_\mathcal{U}$-variety has a
$K_\mathcal{U}v_\mathcal{U}$-rational point.
So by Proposition \ref{prop:Lring-params-henselian-comparison} and Lemma \ref{lem:main}(a,c), 
$$
 \bigcap_\mathcal{U}{\rm Th}_{\forall/\exists}(K_\mathcal{U})=
 \bigcap_\mathcal{U} T_3({\rm Th}(K_\mathcal{U}v_\mathcal{U}))_{\forall/\exists} = 
 T_3(\bigcap_\mathcal{U} {\rm Th}(K_\mathcal{U}v_\mathcal{U}))_{\forall/\exists},
$$
where now $\forall/\exists$ denotes universal/existential $\Lring(K)$-consequences.
As $\bigcap_\mathcal{U}{\rm Th}(K_\mathcal{U}v_\mathcal{U})=\bigcap_\mathcal{U}{\rm Th}(\prod_{v\in\mathbb{P}_K}Kv/\mathcal{U})$
is the $\mathfrak{L}_{\rm ring}(K)$-theory
of almost all $Kv$,
which by Lemma \ref{lem:psf} is the
$\mathfrak{L}_{\rm ring}(K)$-theory
of pseudofinite fields containing $K$,
the first claim follows.
The decidability follows
from Lemma \ref{lem:main}(b) and Lemma \ref{lem:psf}.
\end{proof}

This in particular proves Theorem \ref{thm:intro_noparam}(d).

\begin{remark}
  Unlike in the other sections, it is essential here that the theory of almost all residue fields of a global function field contains the theory of PAC fields.
  The assumption on rational points in Proposition \ref{prop:Lring-params-henselian-comparison} could be dropped if we were willing to assume some variant of resolution of singularities,
  but under \Rfour\ we obtain stronger results in any case (see Section \ref{sec:fullparam}).
\end{remark}

\begin{remark}
Similar to
Remark \ref{rem:all_implies_individual},
we are not able to prove the decidability of the universal/existential $\mathfrak{L}_{\rm ring}(K)$-theory of {\em all} $\hat{K}_v$ as this would imply the decidability of the existential $\mathfrak{L}_{\rm ring}(K)$-theory of each individual $\hat{K}_{v_0}$
(instead of $z^{-1}\notin\mathcal{O}$  use a universal/existential $\mathfrak{L}_{\rm ring}(K)$-sentence that determines the size of the relative algebraic closure of $k$ in the field, and the intersection of the field with a suitably large finite extension of $K$),
which in turn would solve the open problem of deciding the
existential $\mathfrak{L}_{\rm val}(K)$-theory of $\hat{K}_{v_0}$,
as the valuation ring in $\hat{K}_{v_0}$
is definable both by an existential and by a universal $\mathfrak{L}_{\rm ring}(K)$-formula,
cf.~\cite[Lemma 4.9]{ADF23}.

On the other hand it seems unlikely that the decidability of the universal/existential $\mathfrak{L}_{\rm ring}(K)$-theory
of almost all $\hat{K}_v$ directly implies the decidability of the 
the universal/existential $\mathfrak{L}_{\rm val}(K)$-theory
of almost all $\hat{K}_v$, cf.~Remark \ref{rem:define_valuation}.
The relevant questions here are:
\end{remark}

\begin{question}
For a global function field $K$,
are there a universal and an existential $\mathfrak{L}_{\rm ring}(K)$-formula that define the valuation ring in almost all $\hat{K}_v$?
Is there even any universal/existential formula that does this?
\end{question}

\section{Existential $\Lval$-theory with parameters}
\label{sec:fullparam}

\noindent
In this final section we return to $\Lval$ instead of $\Lring$.
We will obtain results conditional on the following weak consequence of resolution of singularities from \cite[p.~2016]{ADF23}, mentioned already in the introduction:
\begin{enumerate}
\item[(R4)] \label{R4} Every field $L$
 which is existentially closed in $L(\!(t)\!)$
is existentially closed in every extension $F/L$ for which there exists a valuation $v$
on $F$ trivial on $L$ with residue field $Fv = L$.
\end{enumerate}
{
In this work we will use \Rfour\ only through results in \cite{ADF23}
that are proven conditionally on it, and therefore we refer the reader to \cite[Proposition 2.3]{ADF23} for how \Rfour\ would follow from resolution of singularities, and to \cite[Remark 2.4]{ADF23}
for fields $L$ for which it is known to hold.
We will make use in particular of the following two main results that are proven conditionally on \Rfour:

\begin{theorem}[{\cite[Proposition 4.11, Corollary 4.16]{ADF23}}]\label{thm:ADF23}
Assume \Rfour.
\begin{enumerate}[$(a)$]
\item Let $(C,u)$ be an equicharacteristic valued field with $uC=\mathbb{Z}$ and $\mathcal{O}_u$ excellent, and let $\pi\in C$ with $u(\pi)=1$.
Let $(F,v)$ and $(F',v')$ be henselian extension of $(C,u)$ such that $v(\pi)$ is minimal positive in $vF$ and $v'(\pi)$ is minimal positive in $v'F'$, and such that $Fv/Cu$ and $F'v'/Cu$ are separable.
If $Fv$ and $F'v'$ have the same existential $\mathfrak{L}_{\rm ring}(Cu)$-theory,
then $(F,v)$ and $(F',v')$ have the same existential $\mathfrak{L}_{\rm val}(C)$-theory.
\item Let $C$ be a field and let 
$T$ be the theory of
equicharacteristic nontrivially valued henselian fields $(F,v)$ extending the trivially valued field $C$, such that $Fv/C$ is separable.
Then $T({\rm Th}_{\forall/\exists}(k))$ is $\exists$-$C$-complete
(cf.~Theorem \ref{thm:AF16_5.7})
for every separable extension $k/C$.
\end{enumerate}
\end{theorem}
}

Fix a field $C$ of characteristic $p>0$ and $t\in C$.
We consider the following property of a valued field $(F,v)$ with $C \subseteq F$:
\begin{enumerate}[label={\rm UR($t$)}]
\item\label{URt} For every $s \in F$, the element $t-s^p \in F$ either has non-positive value under $v$,
  or it is a uniformiser.
\end{enumerate}
Here, 
a {\em uniformiser} for $v$ is an element of $F$ of minimal positive value.
Note that property \URt\ is axiomatised by the universal $\Lval(C)$-sentence
\[ \forall s, r \big( r \not\in \mathcal{O} \rightarrow (t-s^p)r^2 \not \in \mathcal{O} \big) .\]

We let $T_4\supseteq T_1$ be the theory
of henselian nontrivially valued fields $(F,v)$
extending the trivially valued field $C$
satisfying \URt.

\begin{proposition}\label{prop:axiomatise-ur}
Assume \Rfour.
Suppose there exists a perfect subfield $C_0\subseteq C$ with $C/C_0(t)$ separable algebraic.
If $(F,v)$ and $(F',v')$ are henselian nontrivially valued fields
extending the trivially valued field $C$ satisfying \URt\
such that $Fv$ and $F'v'$ have the same universal/existential $\mathfrak{L}_{\rm ring}(C)$-theory,
then $(F,v)$ and $(F',v')$ have the same universal/existential $\mathfrak{L}_{\rm val}(C)$-theory,
i.e.~{\rm(\ref{eqn:star})} holds for $T=T_4$.
\end{proposition}

\begin{proof}
If $t$ does not have a $p$-th root in $Fv$,
  and hence no $p$-th root in $F'v'$,
  then $Fv/C$ and $F'v'/C$ are separable,
  and in this case the claim follows from 
  {Theorem~\ref{thm:ADF23}(b)}
  (using \Rfour).
So let us suppose that $t$ is a $p$-th power in $Fv$ and therefore also in $F'v'$.
The identity on the perfect field $C_0$
extends to 
sections $\zeta \colon Fv \to F$ and $\zeta' \colon F'v' \to F'$ of the residue homomorphisms \cite[Proposition 4.5]{ADF23}.
 Let $s := \zeta(t^{1/p})$ and $s' := \zeta'(t^{1/p})$,
  so that $v(t-s^p) > 0$, $v'(t-s'^p) > 0$.
  By \URt, $\pi := t-s^p$ and $\pi' := t - s'^p$ are uniformisers for $v$ and $v'$, respectively.
  As $Fv$ and $F'v'$ have the same universal/existential $\mathfrak{L}_{\rm ring}(C)$-theory,
  the relative algebraic closures $D$ and $D'$ of $C$ in $Fv$ respectively $F'v'$ are isomorphic (see again \cite[Lemma 20.6.3(b)]{FJ}),
  and we fix once and for all an isomorphism and identify them.
  As $D$ has imperfect exponent at most $1$
  (since it can be no higher than that of $C$)
  the extensions $Fv/D$ and $F'v'/D$ are separable \cite[Lemma 2.7.5]{FJ}.
  The $C_0$-isomorphism $\zeta'\circ\zeta^{-1}|_{\zeta(D)}\colon\zeta(D)\rightarrow\zeta'(D)$ extends uniquely to an isomorphism $g\colon\zeta(D)(\pi)\rightarrow\zeta'(D)(\pi')$ with $g(\pi)=\pi'$.
  As $v$ is trivial on $\zeta(D)$ and $v(\pi)>0$, $v|_{\zeta(D)(\pi)}$ is the $\pi$-adic valuation on the function field $\zeta(D)(\pi)$, and similarly for $v'|_{\zeta'(D)(\pi')}$.
  Therefore, $g$ is in fact an isomorphism of valued subfields of $(F,v)$ respectively $(F',v')$.
  As $g(s)=s'$ and $g(\pi)=\pi'$,  we have $g(t)=t$,
  so $g$ is even a $C_0(t)$-isomorphism.

  Now 
  {Theorem~\ref{thm:ADF23}(a)}
   implies (using \Rfour)
  that the universal/existential $\Lval$-theories of $(F,v)$ and $(F',v')$ agree,
  with parameters from $\zeta(D)(\pi)$ and $\zeta'(D)(\pi')$ respectively,
  where the parameter fields are identified via $g$.
  Let $c \in C$ be arbitrary, and let $f \in C_0(t)[x]$ be its minimal polynomial over $C_0$.
  Now $c$ is the unique zero of $f$ in $F$ with the same residue as $\zeta(c) \in \zeta(D)$.
  This identifies $c$ as the unique witness in $(F,v)$
  of the quantifier-free $\Lval(\zeta(D)(\pi))$-formula 
  $$
   f(x)=0\wedge (x-\zeta(c))\pi^{-1}\in\mathcal{O}.
  $$ 
  It follows that the universal/existential $\Lval(C)$-theories of $(F,v)$ and $(F',v')$ agree,
  since parameters from $C$ may be eliminated in favour of parameters
  from $\zeta(D)(\pi)$ and $\zeta'(D)(\pi')$, respectively.
\end{proof}

\begin{remark}\label{rem:axiom-ur-needed}
When $Fv/C$ and $F'v'/C$ are separable, Proposition \ref{prop:axiomatise-ur} follows from \cite[Proposition 4.16]{ADF23}.
  In this case, axiom \URt\ is automatically satisfied, since $v(t-s^p) \leq 0$ for all $s$.
  Otherwise, the axiom cannot be dispensed with.
  Indeed, let $s \in F$ with $v(t-s^p) > 0$, so $t-s^p$ is a uniformiser by assumption.
  Then the extension field $F' := F(\sqrt{t-s^p})$ with the unique extension of $v$
  is totally ramified and hence also has residue field $Fv$,
  but the $\Lval(C)$-theories of $F$ and $F'$ differ
  since the universal statement \URt\ holds in $F$ but not in $F'$
  (take $s$ as above and $r : = \sqrt{t-s^p}$).
\end{remark}

\begin{remark}\label{rem:axiom-ur-differentials}
  For a valued field $(F,v)$ with $F \supseteq C$ such that $t \in \mathcal{O} := \mathcal{O}_v$,
  axiom \URt\ is equivalent to the assertion that the element $1 \otimes dt \in Fv \otimes_{\mathcal{O}} \Omega_{\mathcal{O}}$
  does not vanish, where $\Omega_{\mathcal{O}}$ is the $\mathcal{O}$-module of (absolute) Kähler differentials.
  Indeed, if \URt\ fails to hold, there exists $s \in F$ such that $t - s^p$ is in the square of the maximal ideal $\mathfrak{m}$ of $\mathcal{O}$.
  It is then easy to show that $1 \otimes dt = 1 \otimes d(t - s^p)$ vanishes in $Fv \otimes_{\mathcal{O}} \Omega_{\mathcal{O}}$.
  Assume conversely that $1 \otimes dt \in Fv \otimes_{\mathcal{O}} \Omega_{\mathcal{O}}$ vanishes.
  We consider the so-called conormal sequence of $Fv$-modules
  \[ 0 \to \mathfrak{m}/\mathfrak{m}^2 \to Fv \otimes_{\mathcal{O}} \Omega_{\mathcal{O}} \to \Omega_{Fv} \to 0 .\]
  (See \cite[Proposition 16.3]{Eisenbud}, and \cite[Corollary 16.13]{Eisenbud} for exactness at the first place;
  note that we write $\Omega_{Fv}$ and $\Omega_{\mathcal{O}}$ where Eisenbud would write $\Omega_{Fv/\mathbb{F}_p}$ and $\Omega_{\mathcal{O}/\mathbb{F}_p}$.)
  We see that $dt \in \Omega_{Fv}$ must vanish,
  and so $t$ has a $p$-th root in $Fv$.
  Let $s \in \mathcal{O}$ with $v(t - s^p) > 0$.
  We have $dt = d(t - s^p)$,
  and so $1 \otimes dt$ is the image of $(t-s^p) + \mathfrak{m}^2$
  under the first map in the conormal sequence.
  By injectivity, it follows that $t - s^p \in \mathfrak{m}^2$,
  so \URt\ is violated.

  This explains the name $\mathrm{UR}(t)$ (ur for unramified):
  The assertion is that there is no ramification witnessed by the differential $dt$.
  We will need the following elementary form of this:
\end{remark}

\begin{lemma}\label{lem:URt}
Let $K/k$ be a global function field of characteristic $p$, $t\in K$ a separating element,
and $v\in\mathbb{P}_K$.
Then \URt\ holds in $\hat{K}_v$ whenever
$v$ is not one of the finitely many places ramifying over the subfield $k(t)$.
\end{lemma}

\begin{proof}
  This could be deduced from the characterisation given in Remark \ref{rem:axiom-ur-differentials},
  but we give a direct argument.
  Assume that $v$ does not ramify over $k(t)$.
  If $v(t) < 0$, this implies that $v(t) = -1$,
  and so for arbitrary $s \in \hat{K}_v$ we have
  $v(t-s^p) = \min(v(t), v(s^p)) \leq -1$ since $v(t) \neq v(s^p)$.
  Assume now that $v(t) \geq 0$.
  The restriction $v|_{k(t)}$ is the place of $k(t)$ associated
  to an irreducible polynomial $f \in k[t]$.
  Since $k[t]/(f)$ is perfect, there exists $g \in k[t]$
  such that $f | t - g^p$.
  As the right-hand side is a separable polynomial, $f^2 \nmid t-g^p$.
  Since $v$ does not ramify over $k(t)$, we deduce that $v(t-g^p) = 1$.
  For arbitrary $s \in \hat{K}_v$, we write $t - s^p = (t-g^p) + (g-s)^p$,
  and observe that the two summands have different value under $v$,
  so $v(t-s^p) = \min(v(t-g^p), v((g-s)^p)) \leq v(t-g^p) = 1$.
  This shows that \URt\ holds in $\hat{K}_v$.
\end{proof}

\begin{theorem}\label{thm:fullparam}
  Assume \Rfour.
  Let $K/k$ be a global function field of characteristic $p$
  and let $t \in K$ be a separating element.
  \begin{enumerate}[$(a)$]
      \item The universal/existential $\Lval(K)$-theory of {\em almost all}  $\hat{K}_v$, $v \in \mathbb{P}_K$,
  is the universal/existential theory of henselian nontrivially valued fields
  extending the trivially valued field $K$,
  satisfying \URt,
  and with residue field pseudofinite.
  \item The universal/existential $\Lval(K)$-theory of {\em all} $\hat{K}_v$, $v \in \mathbb{P}_K$,  is the universal/existential theory of henselian nontrivially valued fields $(L,w)$
  containing $K$ with residue field finite or pseudofinite
  such that \URt\ holds unless
  $w(z_{u_i})>0$ for some $i=1,\dots,r$,
  and such that ${\rm Th}_{\forall/\exists}(L,w)={\rm Th}_{\forall/\exists}(\hat{K}_v,\hat{v})$ whenever $w(z_v)>0$ for some $v \in \mathbb{P}_K$,
  where for each $v\in\mathbb{P}_K$, $z_v\in K$ is chosen to have a zero only at $v$,
  and $u_1,\dots,u_r$ is the list of places of $K$ which ramify over $k(t)$.
  \end{enumerate}
\end{theorem}

\begin{proof}
Let $C=K$.
By Lemma \ref{lem:almost_all_hatKv}
and Lemma \ref{lem:URt},
any model of the $\mathfrak{L}_{\rm val}(K)$-theory of almost all $\hat{K}_v$ 
is a model of $T_4$.
Therefore the proof of (a) works completely analogous to the proof of Theorem \ref{prop:Lring-params-aa},
replacing $T_3$ by $T_4$
and Proposition \ref{prop:Lring-params-henselian-comparison} by Proposition \ref{prop:axiomatise-ur}.

Taking (a) into account, part (b) asserts that a field $(L,w)$ satisfies the universal/existential $\Lval(K)$-theory of all $\hat{K}_v$, $v \in \mathbb{P}_K$,
if and only if $(L,w)$ either satisfies the universal/existential $\Lval(K)$-theory of \emph{almost} all $\hat{K}_v$,
or ${\rm Th}_{\forall/\exists}(L,w) = {\rm Th}_{\forall/\exists}(\hat{K}_v, \hat{v})$ for some $v \in \mathbb{P}_K$.
This is true on general grounds.
Indeed, suppose that $(L,w)$ does not satisfy the universal/existential $\Lval(K)$-theory of almost all $\hat{K}_v$,
so that $(L,w) \models \neg\varphi$ for some universal/existential $\Lval(K)$-sentence $\varphi$ which is true in all $\hat{K}_v$ except for finitely many exceptional places $v_1, \dotsc, v_n$.
If $(L,w)$ also does not satisfy the universal/existential $\Lval(K)$-theory of any individual $(\hat{K}_v, \hat{v})$, we in particular have universal/existential $\Lval(K)$-sentences $\varphi_1, \dotsc, \varphi_n$ such that $(L,w) \models \neg\varphi_i$ but $(\hat{K}_{v_i}, \hat{v}_i) \models \varphi_i$ for every $i = 1, \dotsc, n$.
Therefore $(L,w) \not\models \varphi \vee \varphi_1 \vee \dotsb \vee \varphi_n$ even though $(\hat{K}_v, \hat{v})$ satisfies the sentence on the right-hand side for every $v\in\mathbb{P}_K$,
and so $(L,w)$ is not a model of the universal/existential theory of all $(\hat{K}_v, \hat{v})$.
Conversely, if $(L,w)$ satisfies either the universal/existential $\Lval(K)$-theory of almost all $\hat{K}_v$, or of one individual $\hat{K}_v$,
it is clear that $(L,w)$ satisfies the common universal/existential $\Lval(K)$-theory of all $\hat{K}_v$.
\end{proof}

\begin{corollary}\label{cor:fullparam}
Assume \Rfour. Let $K/k$ be a global function field.
The $\mathfrak{L}_{\rm val}(K)$-theory of {\em almost all} $\hat{K}_v$ and
the
$\mathfrak{L}_{\rm val}(K)$-theory of {\em  all} $\hat{K}_v$
are decidable.
\end{corollary}

\begin{proof}
Fix a separating element $t\in K$
and let $P_0\subseteq\mathbb{P}_K$
be the finite set of places that ramify over $k(t)$.
Similar to before,
the decidability of the 
universal/existential
$\mathfrak{L}_{\rm val}(K)$-theory of {\em almost all} $\hat{K}_v$
follows immediately from
Proposition \ref{prop:axiomatise-ur}, Lemma \ref{lem:psf} and Lemma \ref{lem:main}(b).
In order to decide whether a given
universal/existential 
$\mathfrak{L}_{\rm val}(K)$-sentence $\varphi$,
holds in {\em all} $\hat{K}_v$ first
check whether it holds in {\em almost all} of them.
If yes, then by Theorem~\ref{thm:fullparam} and the completeness theorem,
there is a proof of $\varphi$ from the set of axioms
for henselian nontrivially valued fields extending the trivially valued field $C=K$ satisfying \URt\ and with residue field pseudofinite.
As this set of axioms is computable (see in particular Lemma \ref{lem:psf}), the set of consequences is computably enumerable, and so by enumerating all proofs from these axioms, we will eventually find a proof of $\varphi$.
Apart from possibly using \URt, 
which holds in $\hat{K}_v$ for all $v\notin P_0$ (Lemma \ref{lem:URt}),
this proof will involve only finitely many constants from $K$ and only finitely many axioms of pseudofinite fields,
and so we can compute a finite set $P_1\subseteq\mathbb{P}_K$ such that $\hat{K}_v\models\varphi$ for every $v\notin P_0\cup P_1$.
To determine whether $\hat{K}_v\models\varphi$ for each of the finitely many $v\in P_0\cup P_1$,
we use \cite[Theorem 4.12]{ADF23},
which (again assuming \Rfour) gives a uniform decidability of each individual ${\rm Th}_{\forall/\exists}(\hat{K}_v,\hat{v})$.
\end{proof}

This proves the remaining parts $(e)$ and $(f)$ of Theorem \ref{thm:intro_noparam}.

\begin{remark}
All algorithms presented in this paper can in fact be made uniform in the global function field $K$, similar to Corollary \ref{cor:Edec}.
For Corollary \ref{cor:fullparam}, this could take the following form:
There exists an algorithm that takes as input a prime power $q$, an absolutely irreducible $f\in\mathbb{F}_q[t,x]$ (say separable in $x$) and a 
universal/existential $\mathfrak{L}_{\rm val}(\mathbb{F}_q[t,x])$-sentence $\varphi$,
and determines whether $\varphi$ holds in all respectively almost all completions
of the global function field $K_f={\rm Quot}(\mathbb{F}_q[t,x]/(f))$,
where the parameters from $\mathbb{F}_q[t,x]$ are interpreted as their residues in $K_f$.
To formally prove  this, we would have to argue that the decidability in Lemma \ref{lem:main}(b) is uniform in $R$, and that all parts of our axiomatisations for a given $K_f$ are computable from $f$,
like in Corollary \ref{cor:fullparam}(b) the finite list of places of $K_f$ which ramify over $\mathbb{F}_q(t)$.
\end{remark}

\begin{remark}\label{rem:finite_ext}
As indicated in the introduction,
Corollary \ref{cor:intro_Ax} about completions of number fields follows from Theorem \ref{thm:intro_Ax} about completions of $\mathbb{Q}$:
Indeed, if $L=\mathbb{Q}[x]/(f)$ is a number field,
then $\mathbb{Q}_p[x]/(f)\cong \mathbb{Q}_p\otimes_\mathbb{Q}L\cong\prod_{w|p}\hat{L}_w$, where $w$ runs over the finitely many places of $L$ over $p$,
cf.~\cite[II §3 Theorem 1]{Serre}.
So an $\mathfrak{L}_{\rm ring}$-sentence $\varphi$ 
holds in all respectively almost all $\hat{L}_w$
if and only if $\varphi_f$ 
holds in all respectively almost all $\mathbb{Q}_p$,
where $\varphi_f$ is a sentence 
expressing that,
if $f=f_1\cdots f_r$ is the decomposition into irreducible factors,
$\varphi$ holds in each of the fields generated by a root of one of the $f_i$.
Note however that this does not immediately reduce Theorem~\ref{thm:intro_noparam} to the special case of the rational function field $K=\mathbb{F}_p(t)$, as 
$\varphi_f$ introduces new quantifiers on top of those of $\varphi$.
For existential $\mathfrak{L}_{\rm ring}(K)$-sentences, this problem can be overcome (see \cite[Proposition 2.1.2]{Dittmann}, and also \cite[§6]{Siegel}),
so that at least part (d) of Theorem \ref{thm:intro_noparam}
could be deduced from the special case  $K=\mathbb{F}_p(t)$,
but this is much less clear for parts (c), (e) and (f) of that theorem.
\end{remark}

\section*{Acknowledgements}
\noindent
{The authors would like to thank the referee for suggestions to improve the exposition of this work.}
A.~F.~would like to thank Sylvy Anscombe for helpful discussions at an early stage of this project.
A.~F.~was funded by the Deutsche Forschungsgemeinschaft (DFG) - 404427454.
{
Both authors would like to thank the Hausdorff Research Institute for Mathematics Bonn,
funded by the DFG under Germany's Excellence Strategy (EXC-2047/1 -- 390685813),
for its hospitality during the trimester programme ``Definability, decidability, and computability''
while this paper was revised.
}

\def\bibfont{\footnotesize}
\bibliographystyle{plain}

\end{document}